\documentclass[12pt]{article}
\usepackage{amsmath,amsthm,amssymb,color,amscd}

\def\seq#1#2#3{#1_{#2},\,\ldots,#1_{#3}}

\def\w{\widetilde}
\def\h{\widehat}
\def\b{\overline}

\def\m{\mathfrak m}

\def\uu{{\underline{u}}}
\def\vv{{\underline{v}}}

\def\tt{{\underline{t}}}
\def\nn{{\underline{n}}}
\def\ww{\underline{w}}

\def\mm{\underline{m}}

\def\1{\underline{1}}
\def\0{\underline{0}}
\def\R{\mathbb R}
\def\P{\mathbb P}
\def\Q{\mathbb Q}

\def\LL{{\cal L}}

\def\Z{\mathbb Z}
\def\Q{\mathbb Q}
\def\C{\mathbb C}
\def\S{\mathbb S}
\def\A{\mathbb A}
\def\D{\mathbb D}
\def\E{\mathbb E}
\def\J{\mathbb J}
\def\AA{{\mathcal A}}
\def\OO{{\mathcal O}}
\def\BB{{\mathcal B}}
\def\CC{{\mathcal C}}

\def\XX{{\mathcal X}}
\def\DD{{\mathcal D}}

\def\oDD{\stackrel{\circ}{\cal D}}

\def\oE{\stackrel{\circ}{E}}

\def\frW{{\mathfrak W}}
\def\frC{{\mathfrak C}}

\newtheorem{theorem}{Theorem}

\newtheorem{corollary}{Corollary}
\newtheorem{lemma}{Lemma}
\newtheorem{proposition}{Proposition}

\theoremstyle{definition}
\newtheorem{definition}{Definition}
\newtheorem{remark}{Remark}

\newtheorem{example}{Example}

\newenvironment{Proof}
{\noindent{\bf Proof\/}}{{ $\Box$}\smallskip\par}

\title{Weil--Poincar\'e series and topology of collections of valuations on rational double points
}

\author{
A.~Campillo,
\and F.~Delgado\thanks{The first two authors
partially supported by MCIN/AEI/10.13039/501100011033 and by ``ERDF A way of making
Europe", grant  PGC2018-096446-B-C21.},
\and S.M.~Gusein-Zade\thanks{
The work of the third author was supported by the
grant 21-11-00080 of the Russian Science Foundation.
} }

\date{}
\begin{document}

\def\eps{\varepsilon}

\maketitle

\begin{abstract}
Earlier it was described to which extent the Alexander polynomial in several variables of an algebraic link in the Poincar\'e sphere determines the topology of the link. It was shown that, except some explicitly described cases, the Alexander polynomial of an algebraic link determines the combinatorial
type of the minimal resolution of the curve and therefore the topology of the corresponding link. The Alexander polynomial
of an algebraic link in the Poincar\'e sphere
coincides with the Poincar\'e series of the corresponding set of curve valuations. The latter one can be defined as an
integral over the space of divisors on the $\E_8$-singularity. Here we consider a similar integral
for rational double point surface singularities
over the space of Weil divisors called the
Weil--Poincar\'e series.
We show that, except a few explicitly described
cases the
Weil--Poincar\'e series of a collection of curve valuations on
a rational double point surface singularity determines the topology of the corresponding link.
We give analogous statements for collections of divisorial valuations.
\end{abstract}

\section{Introduction}\label{sec:Introduction}
An algebraic link in the three-dimensional sphere is the intersection $K=C\cap \S_{\eps}^3$ of a germ
$(C,0)\subset (\C^2,0)$ of a complex analytic plane curve with the sphere $\S_{\eps}^3$ of radius $\eps$
centred at the origin in $\C^2$ with $\eps$ small enough. The number $r$ of the components of the link $K$ is equal
to the number of the irreducible components of the curve $(C,0)$. It is well-known that the Alexander polynomial
in $r$ variables determines the topological type of an algebraic link (or, equivalently, the (local) topological
type of the triple $(\C^2,C,0)$): \cite{Yamamoto}. This follows from the fact that the Alexander polynomial
determines the combinatorial type of the minimal embedded resolution of the curve $C$.
The Alexander polynomial is defined for links in three-dimensional manifolds which are homology spheres.
One of them is the Poincar\'e sphere which is the intersection of the surface
$S=\{(z_1,z_2,z_3)\in\C^3: z_1^5+z_2^3+z_3^2=0\}$ (the $\E_8$ surface singularity) with the 5-dimensional sphere
$\S_{\eps}^5=\{(z_1,z_2,z_3)\in\C^3: \vert z_1\vert^2+\vert z_2\vert^2+\vert z_3\vert^2=\eps^2\}$.
An algebraic link in the Poincar\'e sphere is the intersection of a germ $(C,0)\subset (S,0)$ of a complex
analytic curve in $(S,0)$ with the sphere $\S_{\eps}^5$ of radius $\eps$ small enough.
In \cite{MZ}, it was analyzed to which extent the Alexander polynomial in several variables of an algebraic link
in the Poincar\'e sphere determines the topological type of the link or rather the combinatorial type of the
minimal embedded resolution of the corresponding curve. It was shown that, in general, the Alexander polynomial
does not determine the combinatorial type of the minimal embedded resolution of the curve defining an algebraic link
in the Poincar\'e sphere. However, if the strict transform of a curve in $(S,0)$ in the minimal resolution of
$(S,0)$ does not intersect the component corresponding to the end of the longest tail in the graph $\E_8$
of the resolution at a smooth point of the exceptional divisor, then the Alexander polynomial of the corresponding
link determines the combinatorial type of the minimal resolution of the curve and therefore the topology of
the corresponding link.

A valuation (discrete of rank one) on the ring $\OO_{V,0}$ of germs of functions on a
complex analytic variety
$(V,0)$ is a function $v: \OO_{V,0}\to \Z_{\ge 0}\cup \{+\infty\}$ such that
\begin{enumerate}
 \item[1)] $\nu(\lambda g)=\nu(g)$ for $\lambda\in\C$, $\lambda\ne 0$;
 \item[2)] $\nu(g_1+g_2)\ge\min(\nu(g_1),\nu(g_2))$;
 \item[3)] $\nu(g_1g_2)=\nu(g_1)+\nu(g_2)$.
\end{enumerate}
We permit a valuation to have the value infinity for a non-zero element. (In this case some authors speak
about ``semivaluations'').

An irreducible curve germ $(C,0)$ in a germ of a complex analytic variety $(V,0)$ defines a valuation $v_C$
on the ring $\OO_{V,0}$ of germs of functions on $(V,0)$ (called a curve valuation).
Let $\varphi:(\C,0)\to(V,0)$ be a parametrization
(an uniformization) of the curve $(C,0)$, that is ${\text{Im\,}}\varphi=(C,0)$ and $\varphi$ is an isomorphism
between punctured neighbourhoods of the origin in $\C$ and in $C$. For a function germ $f\in\OO_{V,0}$,
the value $v_C(f)$ is defined as the degree of the leading term in the Taylor series of the function
$f\circ\varphi:(\C,0)\to\C$:
$$
f\circ\varphi(\tau)=a\tau^{v_C(f)}+{\text{\ terms of higher degree,}}
$$
where $a\ne 0$; if $f\circ\varphi\equiv0$, one defines $v_C(f)$ to be equal to $+\infty$.

A collection $\{(C_i,0)\}$ of irreducible curves in $(V,0)$, $i=1,\ldots,r$,
defines the collection $\{v_{C_i}\}$ of valuations.
To a collection $\{v_i\}$ of discrete rank one valuations on $\OO_{V,0}$, $i=1,\ldots,r$,
one may associate, as in \cite{CDK}, a Poincar\'e series
$P_{\{v_i\}}(t_1,\ldots,t_r)\in\Z[[t_1,\ldots, t_r]]$.
The collection $\{v_i\}$ defines a multi-index filtration on $\OO_{V,0}$ by
\begin{equation}\label{eq:filtration}
 J(\uu) = \{g\in \OO_{V,0} : \vv(g)\ge \uu\}\; ,
\end{equation}
where
$\uu=(u_1,\ldots, u_r)\in \Z_{\ge 0}^r$,
$\vv(g)=(v_1(g),\ldots,v_r(g))$
and $\uu'=(u'_1, \ldots, u'_r)\ge \uu = (u_1,\ldots, u_r)$ if and only if $u'_i\ge
u_i$ for all $i=1,\ldots,r$.
Equation~(\ref{eq:filtration}) defines the subspaces $J(\uu)$ for all $\uu\in \Z^r$.
The Poincar\'e series of the filtration $\{J(\uu)\}$ (or of the collection
$\{v_i\}$ of valuations) is defined by:
\begin{equation}\label{eq:CDK}
 P_{\{v_i\}}(t_1, \ldots, t_r)=
 \frac{\LL(t_1, \ldots, t_r)\cdot\prod_{i=1}^r(t_i-1)}{t_1\cdot\ldots\cdot t_r-1}\; ,
\end{equation}
where
$$
\LL(t_1, \ldots, t_r):=
\sum\limits_{\uu\in\Z^r} \dim(J(\uu)/J(\uu+\1))\cdot\tt^{\,\uu}\, ,
$$
$\1=(1, 1,\ldots, 1)\in\Z^r$. This definition makes sense if and only if all the quotients
$J(\uu)/J(\uu+\1)$ are finite-dimensional.

In~\cite{Duke} it was shown that, for $(V,0)=(\C^2,0)$, the Poincar\'e series $P_{\{v_{C_i}\}}(t_1,\ldots,t_r)$ of
a collection of (different) curve valuations coincides with the Alexander polynomial $\Delta^C(t_1,\ldots,t_r)$ in
$r$ variables of the algebraic link defined by the curve $C=\bigcup\limits_{i=1}^rC_i$ for $r>1$. (For $r=1$ one has
$P_{v_C}(t)=\frac{\Delta_C(t)}{1-t}$.) In~\cite{CMH} it was shown that the same holds for an algebraic link in
the Poincar\'e sphere. In~\cite{FAOM} it was proved that the Poincar\'e series of a collection of divisorial
valuations on $\OO_{\C^2,0}$ (computed in \cite{DG}) determines the combinatorial type of the minimal resolution
of the collection. (In general, this is not the case for a collection consisisting both of curve and divisorial
valuations.)

In~\cite{MZ}, it was discussed to which extent the Alexander polynomial in several variables of an algebraic link
in the Poincar\'e sphere (that is the Poincar\'e series of the corresponding curve) determines the topology of
the link or rather the combinatorial type of the minimal (embedded) resolution of the curve on the $\E_8$ surface
singularity. It was shown that two curves (even irreducible ones) with combinatorially different minimal
resolutions may have equal Alexander polynomials. However, if the strict transform of a curve does not
intersect one particular component of the exceptional divisor of the minimal resolution of the $\E_8$ surface
singularity, namely the one corresponding to the end of the longest tail in the corresponding $\E_8$-diagram,
then its Poincar\'e series (that is the Alexander polynomial of the corresponding link) determines the combinatorial
type of the minimal resolution of the curve and therefore the topology of the corresponding link. There were
given analogues of these statements for collections of divisorial valuations on the $\E_8$ surface singularity.

For other surface singularities (say, for rational ones, for whom one has a formula for the Poincar\'e series)
the (classical) Poincar\'e series for a collection of curve or divisorial valuations
does not determine the
combinatorial type of the minimal resolution even in the simplest case of the $\A_k$ singularities. The Poincar\'e
series of a collection of valuations can be interpreted as an integral with respect to the Euler characteristic
over the space of Cartier divisors (appropriately defined: see below). We consider an analogue of the Poincar\'e
series which is the same integral over the space of all (that is Weil) divisors. One has an equation for this
(``Weil--Poincar\'e'') series similar to the one for the smooth case or for the $\E_8$ surface singularity.
(The Weil--Poincar\'e series of a collection of curve or divisorial valuations is a fractional power series.)
We show that, except a few cases (somewhat similar to the exception for the $\E_8$ singularity), the
Weil--Poincar\'e series of a collection of curve valuations or of a collection of divisorial valuations on
a rational double point surface singularity determines the combinatorial type of the minimal resolution
(and thus the topology of the corresponding link in the curve case) up to the possible symmetry of the
minimal resolution graph of the surface singularity. (It is somewhat curious that
exceptions exist for the $\E_7$ and for the $\E_8$
surface singularities, i.e.\ precisely for those whose minimal resolution graphs have no non-trivial symmetries.)

\section{The Weil--Poincar\'e series}\label{sec:Definition}
Let $(S,0)$ be a normal surface singularity and let $v_i$, $i=1, \ldots, r$, be either a curve valuation
on $\OO_{S,0}$ defined by an irreducible curve $(C_i,0)\subset(S,0)$, or a divisorial valuation on $\OO_{S,0}$
defined by a component of the exceptional divisor $\DD$ of a resolution $\pi:(\XX,\DD)\to(S,0)$ of the surface $S$.
Assume that $\pi:(\XX,\DD)\to(S,0)$ is a resolution of the surface $S$ which is, at the same time, a resolution
of the collection $\{v_i\}$ of valuations, that is the total transform of the union of the curves $C_i$ (such that
$v_i$ is the curve valuation $v_{C_i}$) is a normal crossing divisor on $\XX$ and each divisor defining
the divisorial valuation from the collection $\{v_i\}$ is present in $\DD$.

Let $\DD=\bigcup\limits_{\sigma\in\Gamma} E_{\sigma}$ be the representation of the exceptional divisor $\DD$
as the union of its irreducible components. For $\sigma\in\Gamma$, let $\oE_{\sigma}$ be the ``smooth part''
of the component $E_{\sigma}$ in the total transform of the curve $\bigcup_i C_i$, i.~e., the component $E_{\sigma}$
itself minus the intersection points with all other components of the exceptional divisor $\DD$ and with
the strict transforms of the curves $C_i$.
A {\em curvette} corresponding to a component $E_{\sigma}$ of the resolution is the
blow-down of
a germ of a smooth curve transversal to $E_{\sigma}$ at a point of $\oE_{\sigma}$.
For $i\in\{1, 2, \ldots, r\}$, let $\tau(i)$ be either the index
of the component $E_{\tau(i)}$ which intersects the strict transform of the curve $C_i$ (if $v_i$ is a curve
valuation), or the index of the component which defines the divisorial valuation $v_i$. Let
$(E_{\sigma}\circ E_{\delta})$ be the intersection matrix of the components $E_{\sigma}$. The diagonal entries
of this matrix are negative integers and a non-diagonal entry is equal to $1$ if the components $E_{\sigma}$ and
$E_{\delta}$ intersect and is equal to $0$ otherwise. Let $(m_{\sigma\delta}):=-(E_{\sigma}\circ E_{\delta})^{-1}$.
The entries $m_{\sigma\delta}$ are positive rational numbers whose denominators divide the determinant $d$ of
the matrix $-(E_{\sigma}\circ E_{\delta})$. For $\sigma\in\Gamma$, let
$\mm_{\sigma}:=(m_{\sigma\tau(1)}, m_{\sigma\tau(2)}, \ldots, m_{\sigma\tau(r)})\in \Q_{>0}^r$.

\begin{definition}\label{def:WPoincare}
 The {\em Weil--Poincar\'e series} ({\em W--Poincar\'e series} for short) of the collection of valuations $\{v_i\}$ is
 \begin{equation}\label{eq:def_W_Poincare}
 P^W_{\{v_i\}}(\tt):=\prod_{\sigma\in\Gamma}\left(1-\tt^{\mm_{\sigma}}\right)^{-\chi(\oE_{\sigma})}\in
 \Z[[t_1^{1/d}, t_2^{1/d}, \ldots, t_r^{1/d}]]\,,
 \end{equation}
 where $\tt=(t_1, t_2, \ldots, t_r)$, for $\mm=(m_1, \ldots, m_r)\in \Q_{\ge0}^r$,
 $\tt^{\mm}=t_1^{m_1}\cdot\ldots\cdot t_r^{m_r}$.
\end{definition}

\begin{remark}
 One can see that the W--Poincar\'e series $P^W_{\{v_i\}}(\tt)$ is well-defined, i.e., does not
 depend on the choice of the resolution $\pi:(\XX,\DD)\to(S,0)$.
 This follows from the fact that a resolution of the collection $\{v_i\}$ can be obtained from the minimal one
 by additional blow-ups either at smooth points of the total transform of the curve $\bigcup_i C_i$ or at
 intersection points of it.
\end{remark}

If $(S,0)$ is smooth or if it is the $\E_8$-surface singularity, the W--Poincar\'e series of the collection
$\{v_i\}$ coincides with the usual Poincar\'e series of $\{v_i\}$ described above: \cite{Duke}, \cite{CMH},
\cite{MZ}.

\begin{remark}
 For the case when $(S,0)$ was a rational surface singularity and $v_i$ were divisorial
 valuations corresponding to all the components of the exceptional divisor of a resolution of $(S,0)$, the series
 $P^W_{\{v_i\}}(\tt)$ was defined in \cite{Invent} and \cite{CMH} and used in \cite{FAA}; see also~\cite{Nemethi}.
\end{remark}


\section{Weil--Poincar\'e series and integrals with respect to the Euler characteristic}\label{sec:Integlal}
In \cite{IJM} it was (essentially) shown that the Poincar\'e series $P_{\{v_i\}}(\tt)$ of a collection $\{v_i\}$
of valuations (curve or divisorial ones) on the ring $\OO_{X,0}$ of germs of functions on a variety $X$ can be
given by the equation
$$
P_{\{v_i\}}(\tt)=\int_{\P\OO_{X,0}}\tt^{\vv}\,d\chi\,,
$$
where the right hand side is the inegral with respect to the Euler characteristic over the projectivization
of $\OO_{X,0}$ (properly defined), $t_i^{\infty}:=0$ (see also \cite[Proposition 1.1]{IJM-2010}).
In~\cite{IJM} it was shown that the Poincar\'e series of a collection of
curve valuations on $\OO_{\C^2,0}$ can be written as an integral with respect to the Euler characteristic over the
configuration space of effective divisors on the smooth part of the exceptional divisors of the embedded
resolution of the union of curves.

Let $(S,0)$ be a normal surface singularity,
let $\{v_i\}$ be a collection of curve or divisorial valuations on $\OO_{S,0}$, and let $\pi:(\XX,\DD)\to(S,0)$
be a resolution of the collection (not the minimal one, in general). Let $E_{\sigma}$, $\sigma\in\Gamma$, be the
irreducible components of the exceptional divisor $\DD$ and let $\oE_{\sigma}$ be the ``smooth part'' of $E_{\sigma}$
in the total transform of the union $(C,0)=\bigcup_i(C_i,0)$ of the irreducible curves $(C_i,0)$ defining curve
valuations from the collection (i.~e.\ $E_{\sigma}$ minus the intersection points with other components of $\DD$
and with the total transforms of the curves $C_i$). Let
$$
Y^{\pi}:=\prod_{\sigma\in\Gamma}\left(\coprod_{k=0}^{\infty}S^k\oE_{\sigma}\right)=
\coprod_{\{k_{\sigma}\}\in \Z^{\Gamma}_{\ge 0}}\prod_{\sigma}S^{k_{\sigma}}\oE_{\sigma}
$$
be the configuration space of effective divisors on $\oDD=\bigcup_{\sigma}\oE_{\sigma}$.

Let $\vv:Y^{\pi}\to \Q_{\ge0}^r$ be the function which sends the component $\prod_{\sigma}S^{k_{\sigma}}\oE_{\sigma}$
of $Y^{\pi}$ to $\sum\limits_{\sigma\in\Gamma}k_{\sigma}\mm_{\sigma}$. Let $\OO_{S,0}^{\pi}$ be the set of
non-zero function germs on $(S,0)$ such that the strict transform of the zero-level curve $\{f=0\}$ intersects $\DD$
only at points of $\oDD$. One has a map $I^{\pi}$ from $\OO_{S,0}^{\pi}$ to $Y^{\pi}$ which sends a function $f$
to the intersection of the strict transform of the curve $\{f=0\}$ with $\oDD$. Let $Y_{\mathcal C}^{\pi}$ be
the image of $I^{\pi}$. The set $\{\pi\}$ of resolutions of the collection $\{v_i\}$ is a partially ordered set:
a resolution is bigger than another one if it can be obtained from the latter by a sequence of blow-ups.

Let ${\frW}_{S,0}$ be the set of effective divisors (that is Weil divisors) on $(S,0)$ and let
$\frC_{S,0}\subset \frW_{S,0}$ be the set of effective Cartier divisors on $(S,0)$. There is a natural map $J$
from $\P\OO_{S,0}$ to $\frC_{S,0}$. For a resolution $\pi$ of the collection $\{v_i\}$, let $\frW^{\pi}_{S,0}$
be the set of divisors in $\frW_{S,0}$ whose strict transforms intersect the exceptional divisor $\DD$ only at
points of $\oDD$. One has the natural map ${\check{I}}^{\pi}$ from $\frW^{\pi}_{S,0}$ onto 
$Y^{\pi}$ and the map $I^{\pi}$ factorizes through it: $I^{\pi}={\check{I}}^{\pi}\circ J$.

Let $\vv:\frW_{S,0}\to \Q^r_{\ge 0}$ be the composition $\vv\circ I$, where
$\vv:Y^{\pi}\to\Q^r_{\ge0}$ is described above. The map $\vv$ sends $\frC_{S,0}$ to $\Z^r_{\ge0}$.
For a rational surface singularity $S$ the map
$\vv:\frW_{S,0}\to \Q^r_{\ge 0}$ can be defined
in the following way. For any divisor $C\in\frW_{S,0}$, a multiple $kC$ of it, $k>0$, is
a Cartier divisor, i.e.\ the divisor of a holomorphic function $f:S\to \C$. (It is
possible to take $k=\det(E_\sigma, E_\delta)$ for a resolution of the singularity $S$.) Then
$\vv(C)=\vv(f)/k$. For two divisors $C$ and $C'$ on $S$ ($C=\bigcup C_i$, $C'=\bigcup C'_j$, where
$C_i$ and $C'_j$ are irreducible curves) the (rational) number $\sum_i
v_{C_i}(C')=\sum_j v_{C'_j}(C)$ can be regarded as the intersection number $(C, C')$
of the curves
$C$ and $C'$ and will be called (and denoted)
in this was. (In these terms the number
$m_{\sigma\delta}$ is the intersection number of curvettes at the components $E\sigma$ and $E_\delta$.)

We shall show that the Poincar\'e series $P_{\{v_i\}}(\tt)$ can be interpreted as an
integral
$$
\int _{\frC_{S,0}}\tt^{\vv}\,d \chi \;
$$
with respect to the Euler characteristic.
Moreover, we shall show that in the same way
the W-Poincar\'e series $P^W_{\{v_i\}}(\tt)$
is equal to
$$
\int _{\frW_{S,0}}\tt^{\vv}\,d \chi \;
$$
For that we have to define such integrals.

To give the definition we shall consider arcs and divisors on $(S,0)$ as arcs and
divisors on a resolution
$\pi: (\XX,\DD)\to (S,0)$. In this case an arc is represented locally (in some local coordinates)
by a pair of power series in a parameter $\tau$.
Let $\LL_{S,0}$ be the space of arcs (that is parametrized curves) on $(S,0)$ and let
$B = \LL_{S,0}/{\text Aut}(\C,0)$ be the space of branches, i.e. non-parametrized arcs. Let
$\BB = \bigsqcup_{k=0}^\infty S^k B$. Each element of $\BB$ represents a (effective)
divisor on $(S,0)$. However some divisors are represented by different elements of
$\BB$. This can be explained by the following situation. Let
$\check\gamma\in B$ be a branch, It is represented by an arc $\gamma=\gamma(\tau)$.
Then the branch $\check\gamma^k$ represented by the arc $\gamma(\tau^k)$ defines the
same divisor as the collection of $k$ copies of $\check\gamma$.
Let us call a branch $\check\gamma$ (represented by an arc $\gamma$)
{\it primitive\/} if $\gamma$ is an uniformization of its image and let
$B_0\subset B$ be the set of primitive branches. One has
$\frW_{S,0} = \bigsqcup_{k=0}^{\infty} S^k B_0$.

Let $J^m B$ be the space of $m$-jets of branches on $(\XX,\DD)$. The restriction of the truncation map to $B_0$
is surjective. Therefore the image of $\bigsqcup_{k=0}^\infty S^k B_0$ in $\bigsqcup_{k=0}^\infty J^m B$ is the
same as the one of $\bigsqcup_{k=0}^\infty B$. This produces a problem to use the image to define the integral
over $\frW_{S,0}$ through this truncation. To avoid this problem, let us consider the subspace $J^m_{prim} B\subset J^m B$
consisting of the jets each representative of whom is primitive (a jet is a class of branches).
Let $\J^m = \bigsqcup_{k=0}^{\infty} S^k J^m_{prim} B$.

Let $w_i: \frW_{S,0}\to \Q_{\ge 0}\cup \{\infty\}$, $i=1,2,\ldots, r$ be functions on
the set of Weil divisors. Let
$w_i^{m}: \J^m\to \Q_{\ge 0}\cup\{\infty\}$ be the function defined by
$w_i^m([a]) =\sup_{a\in [a]} w_i(a)$ where
$[a]\subset \bigsqcup_{k=0}^{\infty} S^k B_0\subset W$ is the equivalence class of
the
$m$-jet $a$.
Let $\ww:= (\seq w1r): \frW_{S,0}\to (\Q_{\ge 0}\cup \{\infty\})^r$ and
$\ww^m:= (\seq {w^m}1r): \J^m \to (\Q_{\ge 0}\cup \{\infty\})^r$. We shall say that
the function $\ww$ is constructible if
$\ww^m$ is constructible for all $m$ (and therefore integrable with respect to the
Euler characteristic).

\begin{definition}
The integral with respect to the Euler characteristic of the function $\tt^{\ww(-)}$
over the space $\frW_{S,0}$ is defined by
\begin{equation}\label{DefInt}
\int\limits_{\frW_{S,0}} \tt^{\ww(-)} d\chi = \lim_{m\to \infty} \int\limits_{
\J^m_{S,0}} \tt^{\ww^{m}(-)} d \chi \in \Z[[t_1^{1/d}, \ldots, t_r^{1/d}]]\; ;
\end{equation}
where the limit in the right hand side is in the sense of the
$\langle \seq t1r\rangle$-adic topology on
$\Z[[t_1^{1/d}, \ldots, t_{r}^{1/d}]]$,
$\langle \seq t1r\rangle$  is the ideal generated by $\seq t1r$.
\end{definition}

If the right hand side of (\ref{DefInt}) makes no sense, i.e. the limit does not exist,
we regard the function $\tt^{\ww(-)}$ as a non-integrable one. For a subset
$\AA\subset \frW_{S,0}$ and a function
$\ww: \AA\to \Z[[t_1^{1/d}, \ldots, t_{r}^{1/d}]]$, the integral
$\int_{\AA}\tt^{\ww(-)} d \chi$ is understood as
$\int_{\frW_{S,0}}\tt^{\h\ww(-)} d \chi$, where
$\h{\ww}(-)$ is the extension of the function $\ww$ by $+\infty$ outside of $\AA$ (recall that
$t^{+\infty}=0$).

Now let $v_i$, $i=1,\ldots,r$ be curve and/or divisorial valuations on $(S,0)$. They define natural maps
(also denoted by $v_i$) from $\frW_{S,0}$ to $\Z[[t_1^{1/d}, \ldots, t_{r}^{1/d}]]$. (In this case one has
$v_i(a+b) = v_i(a)+v_i(b)$.) Moreover, in this case one assumes that $d = \det(-(E_{\sigma}\cdot E_{\delta}))$.

Let $\J^m_\pi$ be the set of jets of divisors which intersect the exceptional divisor $\DD$ of the resolution
$\pi$ only at points of $\oDD$. There is a natural map from $\J^m_{\pi}$ to $Y^{\pi}$. One can see that the
preimage of a point from $Y^\pi$ has the Euler characteristic equal to $1$.
This follows from the following arguments. An arc at a point of $\oDD$ in some local coordinates $(u,v)$ such
that $\oDD$ is given by the equation $u=0$ can be written as $u=\tau^s$, $v$ is a (truncated) series in $t$.
For a fixed $s$ the set of jets of these arcs has the Euler characteristic equal to $1$ (being isomorphic to a
complex affine space). If $s=1$, the jet belongs to $J^{\cdot}_{prim}B$. If $s>1$, the set of jets not belonging
to $J^{\cdot}_{prim}B$ has the Euler characteristic equal to $1$ (being also isomorphic to a complex affine space).
Therefore the set of jets belonging to $J^{\cdot}_{prim}B$ has the Euler characteristic equal to $0$. The
preimage of a point from $Y^{\pi}$ is the union of products of symmetric powers of
these spaces.
All of them have the Euler characteristics equal to zero
except the product of the symmetric powers of the
spaces of the spaces of arcs with $s=1$ whose Euler characteristic is equal to $1$.

The fact that the preimage of a point from $Y^\pi$ has the Euler characteristic equal to $1$ (alongside with
the Fubini formula) implies the following statements.

\begin{proposition}\label{prop:Poincare}
$$
\int\limits_{\frC_{S,0}} \tt^{\vv(-)} d \chi =
\lim_{\{\pi\}} \int\limits_{Y^\pi_{\CC}} \tt^{\vv(-)} d \chi = P_{\{v_i\}}(\tt)\,.
$$
\end{proposition}

\begin{proposition}\label{prop:Weil-Poincare}
$$
\int\limits_{\frW_{S,0}} \tt^{\vv(-)} d \chi =
\lim_{\{\pi\}} \int\limits_{Y^\pi} \tt^{\vv(-)} d \chi =
\prod_{\sigma\in \Gamma} (1-\tt^{\mm_{\sigma}})^{-\chi(\oE_{\sigma})} \; .
$$
\end{proposition}

\begin{corollary}
$$
 P^W_{\{v_i\}}(\tt)=\int\limits_{\frW_{S,0}} \tt^{\vv(-)} d \chi\,.
 $$
\end{corollary}

\section{Curves and divisors on the $\E_7$-singularity whose Weil--Poincar\'e series
do not determine the minimal resolution}\label{sec:Examples}
\begin{example}
 The dual graph of the minimal resolution of the $\E_7$-singularity is shown on Figure~\ref{fig:E_7}. Let $C'$
 be the curvette at the component $E_2$ and let $C''$ be the blow down of a smooth curve ${\widetilde{C}}''$
 on the surface of the resolution tangent to the component $E_7$ at a smooth point (i.~e.\ not at the intersection
 point with $E_6$) with the intersection multiplicity equal to 2 (i.~e.\ the tangency of ${\widetilde{C}}''$
 and $E_7$ is simple). The minimal resolution of the curve $C'$ coincides with the minimal resolution of the
 surface. The dual graph of the minimal resolution of the curve $C''$ is shown on Figure~\ref{fig:resolution_second}.
\begin{figure}
$$
\unitlength=0.50mm
\begin{picture}(80.00,40.00)(0,0)
\thinlines
\put(-15,30){\line(1,0){75}}
\put(60,30){\line(0,-1){30}}
\put(60,15){\vector(1,0){10}}
\put(-15,30){\circle*{2}}
\put(0,30){\circle*{2}}
\put(15,30){\circle*{2}}
\put(30,30){\circle*{2}}
\put(45,30){\circle*{2}}
\put(60,30){\circle*{2}}
\put(60,15){\circle*{2}}
\put(60,0){\circle*{2}}
\put(15,30){\line(0,-1){15}}
\put(15,15){\circle*{2}}
\put(-16,33){{\scriptsize$1$}}
\put(-1,33){{\scriptsize$2$}}
\put(14,33){{\scriptsize$3$}}
\put(29,33){{\scriptsize$5$}}
\put(44,33){{\scriptsize$6$}}
\put(59,33){{\scriptsize$7$}}
\put(55,14){{\scriptsize$9$}}
\put(55,-1){{\scriptsize$8$}}
\put(17,15){{\scriptsize$4$}}
\end{picture}
$$
\caption{The minimal resolution graph of the curve $C''$.}
\label{fig:resolution_second}
\end{figure}
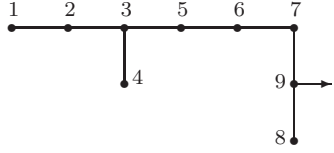
 One can show that
 $$
 P^W_{C'}(t)=P^W_{C''}(t)=\frac{(1-t^{6})(1-t^{8})}{(1-t^{2})(1-t^{3})(1-t^{4})}\,.
 $$
 (The data for these computations and for those in the next example can be taken from the
 matrix~\ref{eq:inverse_for_E_7} on page~\pageref{eq:inverse_for_E_7}.) Therefore the W--Poincar\'e series
 of a curve on the $\E_7$ surface singularity does not determine, in general, the combinatorial type of its
 minimal resolution.
 \end{example}

 \begin{example}
 Let $D'$ be the divisor created by the blow-up of a smooth point of the component $E_9$ of the resolution shown on
 Figure~\ref{fig:resolution_second} and let $D''$ be the divisor created after 3 blow-ups starting at a smooth point
 of the component $E_2$ and produced at each step at a smooth point of the previously created divisor. One can show
 that for the divisorial valuations $v'$ and $v''$ defined by the divisors $D'$ and $D''$ respectively one has
 $$
 P^W_{v'}(t)=P^W_{v''}(t)=\frac{(1-t^{6})(1-t^{8})}{(1-t^{2})(1-t^{3})(1-t^{4})(1-t^9)}\,.
 $$
 Therefore the W--Poincar\'e series of a divisorial valuation on the $\E_7$-singularity does not determine,
 in general, the combinatorial type of the minimal resolution.
 \end{example}

\section{Main statements}\label{sec:Main_res}
Let $(S,0)$ be a rational double point and let $\{v_i\}$, $i=1,\ldots, r$, be a collection of different curve
valuations on $\OO_{S,0}$ defined by irreducible curves $(C_i,0)\subset (S,0)$.

Let us make the following aditional assumptions. If the singularity $(S,0)$ is of type $\E_7$ (respectively of type
$\E_8$), we either assume that the (minimal) resolution process does not contain a blow-up at a smooth point of the
exceptional divisor of the minimal resolution of $(S,0)$ lying on the component $E_7$ (on the component $E_8$
respectively) or assume that does not contain a blow-up at the similar point lying on the component $E_2$
(on the component $E_6$ respectively).

\begin{theorem}\label{theo:Main_curves}
Under the previous assumptions the W--Poincar\'e series
 $P^W_{\{C_i\}}(\tt):=P^W_{\{v_i\}}(\tt)$, $\tt=(t_1,\ldots, t_r)$, of the
curve $C= \bigcup_{i=1}^r C_i$ determines, up to the symmetry of the dual graph of
the minimal resolution of $(S,0)$, the combinatorial type of the minimal (embedded)
resolution of the curve
$(C,0)=\bigcup_{i=1}^r (C_i,0)$
 and therefore the topological type of the link $C\cap L$ in $L=S\cap \S^5_{\eps}$,
i.~e.\ the topological type
 of the pair $(L, C\cap L)$.
\end{theorem}

Now let $\{v_i\}$, $i=1,\ldots,r$, be a collection of different divisorial valuations
on $\OO_{S,0}$. In the cases of $\E_7$ and of $\E_8$ singularities we assume the same
restrictions on the resolution process of the collection of divisorial valuations
as for curves above.

\begin{theorem}\label{theo:Main_divisorial}
Under the previous assumptions the W--Poincar\'e series
$P^W_{\{v_i\}}(\tt)$ determines, up to the symmetry of the dual resolution graph of
$(S,0)$, the combinatorial type of the minimal resolution of the collection $\{v_i\}$.
\end{theorem}

\begin{remark}
 To a divisorial valuation $v_i$ one can associate a curve $C_i$: a curvette at the component
 $E_{\tau(i)}$ defining the valuation $v_i$. Theorem~\ref{theo:Main_divisorial} implies that the series
 $P^W_{\{v_i\}}(\tt)$ determines the topological type of the link $\left(\bigcup_{i=1}^r C_i\right)\cap L$ in $L$.
\end{remark}

\begin{remark}
 In a statement like Theorems~\ref{theo:Main_curves} and~\ref{theo:Main_divisorial} one cannot mix
 curve and divisorial valuations in one collection; see an example in \cite{FAOM}.
\end{remark}

\section{The case of one valuation}\label{sec:One_valuation}
Let $(S,0)$ be a rational double point (of type $\A_k$, $\D_k$, $\E_6$, $\E_7$, or $\E_8)$)
 and let $v$ be either a curve valuation (defined by a curve germ
$(C,0)\subset(S,0)$) or a divisorial valuation on
$\OO_{S,0}$. In the latter case, let $(C,0)\subset(S,0)$
be a curvette at the divisor defining the valuation.
(A resolution of a divisorial valuation is at the same time
a resolution of the corresponding curvette, but not vise versa.) The minimal resolution of the valuation $v$ is obtained from the minimal resolution of
the surface $(S,0)$ by a sequence of blow-ups made (at each step) at intersection points of the strict transform
of the curve $C$ and the exceptional divisor.
Let $\pi':(\XX',\DD')\to(S,0)$ be the minimal resolution of the
surface $(S,0)$ such that the strict transform of the curve $C$ intersects the exceptional divisor $\DD'$
at its smooth point. This resolution is obtained from the minimal resolution of $(S,0)$ by blow-ups made
(at each step) at intersection points of the components of the exceptional divisor.

\begin{definition}\label{def:pre-resol}
 The resolution $\pi':(\XX',\DD')\to(S,0)$ of $(S,0)$ will be called the {\em pre-resolution} of the valuation $v$.
\end{definition}

Let $E_{\sigma_0}$ be the component of the exceptional divisor $\DD'$ of the
pre-resolution $\pi'$ intersecting the strict transform of curve $C$ and let $\ell$
be the intersection number of them in the space $\XX'$ of the pre-resolution. We
shall use
the numbering of the components of the exceptional divisor of the minimal
resolution of $(S,0)$ shown on Figures \ref{fig:A_k}, \ref{fig:D_k}, \ref{fig:E_6},
\ref{fig:E_7}, and \ref{fig:E_8} below.
(They are at the same time components of the exceptional divisor of the pre-resolution $\pi'$.)
In the case of the $\E_7$ ($\E_8$) singularity we either assume that $\sigma_0$ is
not $7$ ($8$ respectively) or assume that it is not $2$ ($6$ respectively).
Pay atention that in all the cases excluded from consideration the pre-resolution
$\pi'$ is the minimal resolution of the surface $(S,0)$.

\begin{lemma}\label{lemma:curve_preresolution}
Under the conditions above, the W--Poincar\'e series
$P^W_C(t)$ determines, up to the symmetry of the
resolution graph of the minimal resolution of $(S,0)$,
the pre-resolution $\pi'$, the component $E_{\sigma_0}$ of the exceptional divisor
$\DD'$, and the intersection multiplicity $\ell$.
\end{lemma}

\begin{proof}
 The proof is based on the analysis of the matrix $(m_{\sigma\delta})$ which has to be made separately for
 different cases. In all the cases, let us write the Poincar\'e series $P^W_C(t)$ in the form
 \begin{equation}\label{eq:curve_binomial_product}
  \prod_{i=1}^q(1-t^{m_i})^{-1}\prod_{m> 0}(1-t^{m})^{s_m},
 \end{equation}
 where $m_1\le m_2\le \ldots\le m_q$ (thus the first product may have repeated factors) and, in the second product,
 the (integer) exponents $s_m$ are non-negative and are equal to zero for $m=m_i$, $i=1,\ldots,q$. Let us recall
 that the representation of the Poincar\'e series in this form is unique.

\medskip
\noindent{\bf Case of $\A_k$ singularity.}
 Let $(S,0)$ be the singularity of type $\A_k$.
The minimal resolution graph is shown on Fig.~\ref{fig:A_k}.
 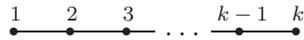
\begin{figure}[h]
 $$
\unitlength=0.50mm
\begin{picture}(80.00,40.00)(0,10)
\thinlines
\put(-15,30){\line(1,0){37.5}}
\put(37.5,30){\line(1,0){22.5}}
\put(25,29){\ldots}
\put(-15,30){\circle*{2}}
\put(0,30){\circle*{2}}
\put(15,30){\circle*{2}}
\put(45,30){\circle*{2}}
\put(60,30){\circle*{2}}
\put(-16,33){{\scriptsize$1$}}
\put(-1,33){{\scriptsize$2$}}
\put(14,33){{\scriptsize$3$}}
\put(39,33){{\scriptsize$k-1$}}
\put(59,33){{\scriptsize$k$}}
\end{picture}
$$
\caption{The dual resolution graph of the $\A_k$-singularity.}
\label{fig:A_k}
\end{figure}

 The matrix $(m_{ij})$ is
  \begin{equation*}\label{eq:inverse_for_A_k}
 \frac{1}{k+1}\cdot
 \begin{pmatrix}
  k & k-1 & k-2 & \cdots & i & \cdots & 2 & 1\\
  k-1 & 2(k-1) & 2(k-2) & \cdots & 2i & \cdots & 4 & 2\\
  k-2 & 2(k-2) & 3(k-2) & \cdots & 3i & \cdots & 6 & 3\\
  \vdots & \vdots & \vdots & \ddots & \vdots & \vdots & \vdots & \vdots\\
  i & 2i & 3i & \cdots & i(k-i+1) & \cdots & 2(k-i+1) & k-i+1\\
  \vdots & \vdots & \vdots & \vdots & \vdots & \ddots & \vdots & \vdots\\
  1 & 2 & 3 & \cdots & k-i+1 & \cdots & k-1 & k
 \end{pmatrix}.
 \end{equation*}

 To identify the components of the exceptional divisor $\DD'$, let us mark them by the indices $\sigma$ being
 rational numbers inbetween $1$ and $k$, naming the component created by the blow-up of the intersection
 point of the components $E_{\sigma_1}$ and $E_{\sigma_2}$ by $E_{\frac{\sigma_1+\sigma_2}{2}}$. (This methods
 can be applied to other rational double points under some restrictions.)
 Since we have to find $\pi'$ and $E_{\sigma_0}$ up to the symmetry of the graph in Fig.~\ref{fig:A_k},
 we can assume that $\sigma_0\ge \frac{k+1}{2}$. We shall consider the following two cases:
 \begin{enumerate}
 \item[1)] $\sigma_0=k$;
 \item[2)] $\frac{k+1}{2}\le\sigma_0<k$.
 \end{enumerate}
 Let us consider Case 1.
 For the curve case either $q=1$ or $\frac{m_2}{m_1}>k$ (the first option takes place if $C$ is a curvette at $E_k$). For the divisorial case one has $\frac{m_2}{m_1}>k$.\newline
 In Case~2 one has $q\ge 2$, $m_1=\ell m_{\sigma_0 1}$, $m_2=\ell m_{\sigma_0 k}$ and therefore $\frac{m_2}{m_1}<k$ (in contrast with
 Case~1).
 The fact that the series $P^W_{\bullet}(t)$ determines the component $E_{\sigma_0}$ follows
 from the fact that the ratio $\frac{m_2}{m_1}$ is strictly growing with $\sigma_0$.\newline
 In both cases the intersection multiplicity $\ell$ is determined by the equation $m_1=\ell m_{\sigma_0 1}$.

 \medskip
\noindent{\bf Case of $\D_k$ singularity.}
 Let $(S,0)$ be the singularity of type $\D_k$. The minimal resolution graph is shown on Fig.~\ref{fig:D_k}.
 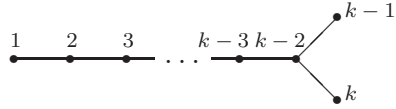
\begin{figure}[h]
 $$
\unitlength=0.50mm
\begin{picture}(80.00,40.00)(0,10)
\thinlines
\put(-15,30){\line(1,0){37.5}}
\put(-15,30){\circle*{2}}
\put(0,30){\circle*{2}}
\put(15,30){\circle*{2}}
\put(45,30){\circle*{2}}
\put(60,30){\circle*{2}}
\put(37.5,30){\line(1,0){22.5}}
\put(25,29){\ldots}
\put(60,30){\line(1,1){11}}
\put(60,30){\line(1,-1){11}}
\put(71,41){\circle*{2}}
\put(71,19){\circle*{2}}
\put(-16,33){{\scriptsize$1$}}
\put(-1,33){{\scriptsize$2$}}
\put(14,33){{\scriptsize$3$}}
\put(34,33){{\scriptsize$k-3$}}
\put(49,33){{\scriptsize$k-2$}}
\put(73,41){{\scriptsize$k-1$}}
\put(73,19){{\scriptsize$k$}}
\end{picture}
$$
\caption{The dual resolution graph of the $\D_k$-singularity.}
\label{fig:D_k}
\end{figure}

 The matrix $(m_{ij})$ is
 \begin{equation*}\label{eq:inverse_for_D_k}
 \frac{1}{4}\cdot
 \begin{pmatrix}
 4 & 4 & 4 & \cdots & 4 & 2 & 2\\
 4 & 8 & 8 & \cdots & 8 & 4 & 4\\
 4 & 8 & 12 & \cdots & 12 & 6 & 6\\
 \vdots & \vdots & \vdots & \ddots & \vdots & \vdots & \vdots\\
 4 & 8 & 12 & \cdots & 4(k-2) & 2(k-2) & 2(k-2)\\
 2 & 4 & 6 & \cdots & 2(k-2) & k & k-2 \\
 2 & 4 & 6 & \cdots & 2(k-2) & k-2 & k
 \end{pmatrix}.
 \end{equation*}

 Assume first that $k>4$. Because of the symmetry of the graph $\D_k$ we can assume that $\sigma_0$ does not
 belong to the lower right tail of the graph, that is $1\le\sigma_0\le k-1$ (we use the same numbering of newly
 created components as above for $\A_k$). We shall consider the following four cases:
 \begin{enumerate}
 \item[1)] $1<\sigma_0\le k-2$;
 \item[2)] $k-2<\sigma_0<k-1$;
 \item[3)] $\sigma_0=1$;
 \item[4)] $\sigma_0=k-1$.
 \end{enumerate}

 In Cases 1 and 2 (both in the curve and the divisorial cases) one has $q\ge 3$ and $m_1$, $m_2$, and $m_3$ are $\ell m_{\sigma_0 1}$, $\ell m_{\sigma_0 k-1}$,
 and  $\ell m_{\sigma_0 k}$ in a certain order.

 In Case 1 at least two of the exponents $m_1$, $m_2$, and $m_3$ coincide
 (and are equal to $m'$). Let us denote the third component by $m''$. We always have $\frac{m'}{m''}>\frac{1}{2}$.
 All three exponents coincide, that is $m''=m'$, if and only if $\sigma_0=2$. If $m''>m'$, then $\sigma_0<2$
 and $m_1=m_2=\ell m_{\sigma_0 k-1}$, $m_3=\ell m_{\sigma_0 1}$. If $m''<m'$, then $2<\sigma_0\le k-2$,
 $m_1=\ell m_{\sigma_0 1}$, $m_2=m_3=\ell m_{\sigma_0 k-1}$. In this case the ratio
 $\frac{m_{\sigma_0 (k-1)}}{m_{\sigma_0 1}}$ is strictly growing with $\sigma_0$ and therefore determines $\sigma_0$.

 In Case 2 the exponents $m_1$, $m_2$, and $m_3$ are different and moreover $m_1=\ell m_{\sigma_0 1}$,
 $m_2=\ell m_{\sigma_0 k}$, $m_3=\ell m_{\sigma_0 (k-1)}$, $\frac{m_3}{m_2}<\frac{5}{4}$. Again the ratio
 $\frac{m_{\sigma_0 (k-1)}}{m_{\sigma_0 1}}$ is strictly gowing and therefore determines $\sigma_0$.

 The equations above determine $\ell$.

 In Case 3 one has: in the curve case either $q=2$ with
 $m_1=m_2=2\ell$
 or $m'=m_1=m_2=2\ell$, $m''=m_3>4\ell$ with $\frac{m'}{m''}<\frac{1}{2}$;
 in the divisorial case $\frac{m'}{m''}\le\frac{1}{2}$,
 $m_1=m_2=2\ell$.

 In Case 4 one has: in the curve case either $q=2$
 with $m_1=2\ell$
 or $m_1$, $m_2$, and $m_3$ are different with $\frac{m_3}{m_2}>\frac{5}{4}$, $m_1=2\ell$;
 in the divisorial case $m_1$, $m_2$, and $m_3$ are different with $\frac{m_3}{m_2}\ge\frac{5}{4}$, $m_1=2\ell$.

 Now let $k=4$, i.~e.\ $(S,0)$ is the singularity of type $\D_4$. Because of the
symmetry of the graph,
 we can assume that $1\le\sigma_0\le 2$. We shall consider the following two case:
 \begin{enumerate}
 \item[1)] $1<\sigma_0\le 2$.
 \item[2)] $\sigma_0=1$;
 \end{enumerate}
 In Case 1 one has (both for the curve and for the divisorial cases) $m_1=m_2=\ell m_{\sigma_0,3}$, $m_3=\ell m_{\sigma_0,1}$ and
$m_3/m_1<2$.
The ratio $\frac{m_{\sigma_0 3}}{m_{\sigma_0 1}}$ is strictly growing with $\sigma_0$ and therefore determines the latter one.
In Case 2 one has: in the curve case either $q=2$
with $m_1=m_2=2\ell$
or $m_1=m_2=2\ell$ and $m_3/m_1>2$;
in the divisorial case $m_1=m_2=2\ell$ and $m_3/m_1\ge 2$.

\medskip
\noindent{\bf Case of $\E_6$ singularity.}
 Let $(S,0)$ be the singularity of type $\E_6$. The minimal resolution graph is shown on Fig.~\ref{fig:E_6}.
 \begin{figure}[h]
 $$
\unitlength=0.50mm
\begin{picture}(80.00,40.00)(0,10)
\thinlines
\put(-15,30){\line(1,0){60}}
\put(-15,30){\circle*{2}}
\put(0,30){\circle*{2}}
\put(15,30){\circle*{2}}
\put(30,30){\circle*{2}}
\put(45,30){\circle*{2}}
\put(15,30){\line(0,-1){15}}
\put(15,15){\circle*{2}}
\put(-16,33){{\scriptsize$1$}}
\put(-1,33){{\scriptsize$2$}}
\put(14,33){{\scriptsize$3$}}
\put(29,33){{\scriptsize$5$}}
\put(44,33){{\scriptsize$6$}}
\put(17,15){{\scriptsize$4$}}
\end{picture}
$$
\caption{The dual resolution graph of the $\E_6$-singularity.}
\label{fig:E_6}
\end{figure}
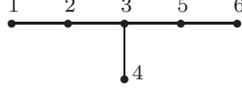

 The matrix $(m_{ij})$ is
 \begin{equation*}\label{eq:inverse_for_E_6}
 \begin{pmatrix}
 4/3&5/3&2&1&4/3&2/3\\
 5/3&10/3&4&2&8/3&4/3\\
 2&4&6&3&4&2\\
 1&2&3&2&2&1\\
 4/3&8/3&4&2&10/3&5/3\\
 2/3&4/3&2&1&5/3&4/3
 \end{pmatrix}.
 \end{equation*}
 Because of the symmetry of the graph, we can assume that $\sigma_0$ does not belong to the right tail of the graph,
 i.~e.\ (using the same rule of numbering of the components of the exceptional divisor $\DD'$ as above)
 $1\le\sigma_0\le 4$. We shall consider the following four cases:
 \begin{enumerate}
 \item[1)] $1<\sigma_0<3$;
 \item[2)] $3\le\sigma_0<4$;
 \item[3)] $\sigma_0=1$;
 \item[4)] $\sigma_0=4$.
 \end{enumerate}

In Cases 1 and 3 one has $m_1<m_2$; in Cases 2 and 4 $m_1=m_2$.

In Case 1 $\frac{m_3}{m_1}<2$. In Case 3 one has:
in the curve case either $q=2$ or $\frac{m_3}{m_2}>2$;
in the divisorial case $\frac{m_3}{m_2}\ge 2$ with $m_1=\ell$ in all the cases.

In Case 2 $\frac{m_3}{m_1}<2$. In Case 4 one has:
in the curve case either $q=2$ or $\frac{m_3}{m_1}>2$;
in the divisorial case $\frac{m_3}{m_1}\ge 2$ with $m_1=\ell$ in all the cases.

In Case 1 (both for the curve and for the divisorial valuation) one has either $\frac{m_2}{m_1}=\frac{3}{2}$ or $\frac{m_3}{m_1}=\frac{3}{2}$.
If $\frac{m_2}{m_1}=\frac{3}{2}$, then $\sigma_0\le 1.5$, the ratio $\frac{m_3}{m_1}$ is strictly increasing with
$\sigma_0$ and therefore determines $\sigma_0$; $m_1=\ell m_{\sigma_0 4}$. If $\frac{m_3}{m_1}=\frac{3}{2}$, then $\sigma_0\ge 1.5$,
the ratio $\frac{m_2}{m_1}$ is strictly decreasing with $\sigma_0$ and therefore determines $\sigma_0$; $m_1=\ell m_{\sigma_0 6}$.
In Case 2 (both for the curve and for the divisorial valuation) the ratio $\frac{m_3}{m_1}$ is strictly increasing with $\sigma_0$ and therefore determines $\sigma_0$; $m_1=\ell m_{\sigma_0 1}$.

\medskip
\noindent{\bf Case of $\E_7$ singularity.}
 Let $(S,0)$ be the singularity of type $\E_7$. The minimal resolution graph is shown on Fig.~\ref{fig:E_7}.
 \begin{figure}[h]
 $$
\unitlength=0.50mm
\begin{picture}(80.00,40.00)(0,10)
\thinlines
\put(-15,30){\line(1,0){75}}
\put(-15,30){\circle*{2}}
\put(0,30){\circle*{2}}
\put(15,30){\circle*{2}}
\put(30,30){\circle*{2}}
\put(45,30){\circle*{2}}
\put(60,30){\circle*{2}}
\put(15,30){\line(0,-1){15}}
\put(15,15){\circle*{2}}
\put(-16,33){{\scriptsize$1$}}
\put(-1,33){{\scriptsize$2$}}
\put(14,33){{\scriptsize$3$}}
\put(29,33){{\scriptsize$5$}}
\put(44,33){{\scriptsize$6$}}
\put(59,33){{\scriptsize$7$}}
\put(17,15){{\scriptsize$4$}}
\end{picture}
$$
\caption{The dual resolution graph of the $\E_7$-singularity.}
\label{fig:E_7}
\end{figure}
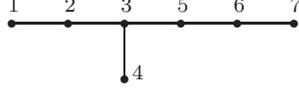

 The matrix $(m_{ij})$ is
 \begin{equation}\label{eq:inverse_for_E_7}
 \begin{pmatrix}
 2&3&4&2&3&2&1\\
 3&6&8&4&6&4&2\\
 4&8&12&6&9&6&3\\
 2&4&6&7/2&9/2&3&3/2\\
 3&6&9&9/2&15/2&5&5/2\\
 2&4&6&3&5&4&2\\
 1&2&3&3/2&5/2&2&3/2
 \end{pmatrix}.
 \end{equation}

Let us analyze first the situation when we assume that $\sigma_0\neq 7$. We shall
consider the following cases:
\begin{enumerate}
\item $\sigma_0\neq 1,4$;
\item $\sigma_0=1$;
\item $\sigma_0=4$.
\end{enumerate}

In Case 1, both for the curve and for the divisorial valuations one has
$m_3=\ell m_{\sigma_0,4}$ and $m_1$ and $m_2$ are $\ell m_{\sigma_0 1}$
and $\ell m_{\sigma_0, 7}$ in a certain order. Moreover
$m_2/m_1<2$ and $m_3/m_1 < 7/3$.

On Figure~\ref{fig:E_7-proj} the ratio $(m_{\sigma_0, 1}:m_{\sigma_0,4}:m_{\sigma_0,7})\in\R\P^2$
is shown (by the bold lines) in the affine chart $(m_{\sigma_0 ,1}/m_{\sigma_0,4}, m_{\sigma_0,7}/m_{\sigma_0,4})$.
 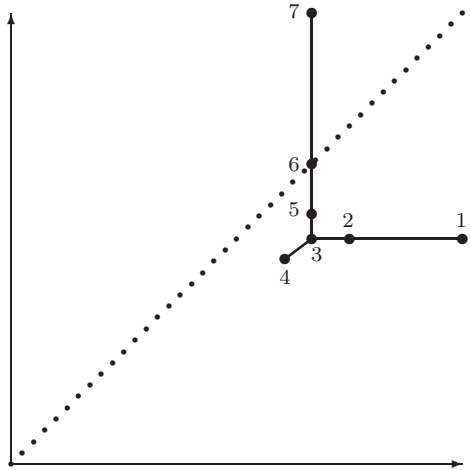
\begin{figure}[h]
 $$
\unitlength=3mm
\begin{picture}(20.00,20.00)(0,0)
\thinlines
\put(0,0){\vector(1,0){20}}
\put(0,0){\vector(0,1){20}}
\thicklines
\put(13.33,10){\line(1,0){6.67}}
\put(13.33,10){\line(0,1){10}}
\put(13.33,10){\line(-4,-3){1.2}}
\put(13.33,10){\circle*{0.5}}
\put(15,10){\circle*{0.5}}
\put(20,10){\circle*{0.5}}
\put(13.33,11.11){\circle*{0.5}}
\put(13.33,13.33){\circle*{0.5}}
\put(13.33,20){\circle*{0.5}}
\put(12.13,9.1){\circle*{0.5}}
\put(19.7,10.5){{\scriptsize 1}}
\put(14.7,10.5){{\scriptsize 2}}
\put(13.3,9){{\scriptsize 3}}
\put(11.9,8){{\scriptsize 4}}
\put(12.3,11){{\scriptsize 5}}
\put(12.3,13){{\scriptsize 6}}
\put(12.3,19.8){{\scriptsize 7}}
\multiput(0,0)(0.5,0.5){41}%
{\circle*{0.2}}
\end{picture}
$$
\caption{The points $(m_{\sigma_0, 1} : m_{\sigma_0,4} : m_{\sigma_0,7})\in \R\P^2$.}
\label{fig:E_7-proj}
\end{figure}
(The fact that the edges meeting at a vertex of valency $2$ lie on a straight line is a general feature of
pictures of this sort.)
The figure shows that the ratio $(m_{\sigma_0, 1} : m_{\sigma_0,4} : m_{\sigma_0,7})$ determines $\sigma_0$.
However the explanation above says that the W-Poincar\'e series determines this ratio
only up to the exchange of the first and the third components. The graph obtained by
exchanging $m_{\sigma_0, 1}$ and $m_{\sigma_0,7}$ is drown by thin lines. One can see
that they intersect only at a point on the diagonal
$\frac{m_{\sigma_0,1}}{m_{\sigma_0,4}} = \frac{m_{\sigma_0,7}}{m_{\sigma_0,4}}$ and
therefore the ratios
$(m_{\sigma_0, 1} : m_{\sigma_0,4} : m_{\sigma_0,7})$ determines $\sigma_0$.

In Case 2, one has $m_2/m_1=2$ (in contrast to Case 1 above and Case 3 below); $m_1=\ell$. In
Case 3 one has $m_2/m_1 =4/3$ and for the curve case,
either $q=2$ or $m_3/m_1 > 7/3$; for the divisorial case
$m_3/m_1\ge 7/3$; $m_2=2\ell$.

In the situation when we assume that $\sigma_0\neq 2$, we shall consider the
following four cases:
\begin{enumerate}
\item $\sigma_0\neq 1,4, 7$;
\item $\sigma_0=1$;
\item $\sigma_0=4$;
\item $\sigma_0=7$.
\end{enumerate}

The analysis of the Cases 1 to 3 is the same as above. In Case 4 one has
$m_2/m_1 =3/2$. This differs Case 4 from Cases 2 and 3 and in Case 1 the value
$m_2/m_1=3/2$ holds only if $\sigma_0=2$. In this case $m_1=\ell$.

\medskip

\noindent{\bf Case of $\E_8$ singularity.}
Let $(S,0)$ be the singularity of type $\E_8$. The minimal resolution
graph is shown in Figure~\ref{fig:E_8}.
 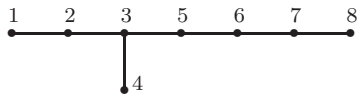
\begin{figure}[h]
 $$
\unitlength=0.50mm
\begin{picture}(80.00,40.00)(0,10)
\thinlines
\put(-15,30){\line(1,0){90}}
\put(-15,30){\circle*{2}}
\put(0,30){\circle*{2}}
\put(15,30){\circle*{2}}
\put(30,30){\circle*{2}}
\put(45,30){\circle*{2}}
\put(60,30){\circle*{2}}
\put(75,30){\circle*{2}}
\put(15,30){\line(0,-1){15}}
\put(15,15){\circle*{2}}
\put(-16,33){{\scriptsize$1$}}
\put(-1,33){{\scriptsize$2$}}
\put(14,33){{\scriptsize$3$}}
\put(29,33){{\scriptsize$5$}}
\put(44,33){{\scriptsize$6$}}
\put(59,33){{\scriptsize$7$}}
\put(74,33){{\scriptsize$8$}}
\put(17,15){{\scriptsize$4$}}
\end{picture}
$$
\caption{The dual resolution graph of the $\E_8$-singularity.}
\label{fig:E_8}
\end{figure}

The matrix $(m_{ij})$ is:
\begin{equation*}\label{eq:inverse_for_E_8}
\begin{pmatrix}
4&7&10&5&8&6&4&2\\
7&14&20&10&16&12&8&4\\
10&20&30&15&24&18&12&6 \\
5&10&15&8&12&9&6&3\\
8&16&24&12&20&15&10&5\\
6&12&18&9&15&12&8&4\\
4&8&12&6&10&8&6&3\\
2&4&6&3&5&4&3&2
\end{pmatrix}
\end{equation*}

The situation when we assume that $\sigma_0\neq 8$ was analyzed in \cite{MZ}. One can see that, in this case,
the possibility to restore $\ell$ easily follows from the discussion therein.
In the situation when we assume that $\sigma_0\neq 6$, we shall consider the following four cases:

\begin{enumerate}
\item $\sigma_0\neq 1,4, 8$;
\item $\sigma_0=1$;
\item $\sigma_0=4$;
\item $\sigma_0=8$.
\end{enumerate}

In Case 1 one has $m_2/m_1\le 2$. The way to determine $\sigma_0$ in this case is
described in \cite{MZ}; $\m_1=\ell m_{\sigma_0 8}$. In Case 2 one has $m_2/m_1 = 5/2$. In Case 3 one has
$m_2/m_1 =5/3$ and either $q=2$ (in the curve case) or $m_3/m_1\ge 8/3$.
This can be met in Case 1 when $3\le \sigma_0 < 4$ (i.e. if $\sigma_0$ is on the
lower tail of the diagram). In this case one has
$5/2 \le m_3/m_1 < 8/3$. In case 3 $m_1=3\ell$.

In Case 4 one has $m_2/m_1=3/2$. This differs Case 4 from Cases 2 and 3. In Case 1
the value $m_2/m_1 =3/2$ holds only if $\sigma_0=6$. In Case 4 $m_1=2\ell$.

\end{proof}

Let $(C,0)$ be an irreducible curve on a rational double point $(S,0)$.

\begin{proposition}\label{prop:main_curves}
Under the described assumptions, the W-Poincar\'e series $P^{W}_{C}(t)$ determines
the
combinatorial type of the minimal embedded resolution of the curve $C$ and therefore
the topological type ot the knot $C\cap L$ in $L=S\cap \S^5_{\varepsilon}$.
\end{proposition}

Let $v$ be a divisorial valuation on $(S,0)$.

\begin{proposition}\label{prop:main_divis}
Under the described assumptions, the W-Poincar\'e series $P^{W}_{\{v\}}(t)$
determines the
combinatorial type of the minimal resolution of the valuation $v$.
\end{proposition}

Proofs of Propositions \ref{prop:main_curves} and \ref{prop:main_divis} are literally the same as of Theorems 1 and 2 in
\cite{MZ}.

 Assume that $C_1$ and $C_2$ are two curves with
 the (known) W--Poincar\'e series $P^W_{C_i}(t)$
 such that the components $E_{\sigma_0^i}$ of the exceptional divisors $\DD'_i$ of
the pre-resolutions emerge from the parts
 (one can say ``tails'' in all the cases but $\A_k$) of the resolution graph of the
minimal resolution of $(S,0)$ exchangable by a symmetry of the graph
 acting non-trivially on them.
 (Pay attention that in this case $S$ is one of
 the singularities $\A_k$, $\D_k$, and $\E_6$.
 It cannot be a singularity of type $\E_7$ or
 $E_8$ when the series W--Poincar\'e series determines the component $E_{\sigma_0}$ of the
 pre-resolution only under additional conditions.)

 \begin{lemma}\label{lemma:two}
In the discribed situation (for $(S,0)$
of the type $\A_k$, $\D_k$, or $\E_6$), The W--Poincar\'e series $P^W_{C_i}(t)$, $i=1.2$,
alongside with the intersection number
$(C_1\circ C_2)$ determines whether the strict
transforms of the curves $C_1$ and $C_2$
intersect the same part of the graph of the minimal resolution
of the surface singularity $(S,0)$ or different ones.
 \end{lemma}

 \begin{proof}
  If the strict transforms intersect the same part of the graph, then
  $(C_1\circ C_2)\ge \ell_1\cdot\ell_2\cdot m_{\sigma_0^1\sigma_0^2}$. (The sign $>$ may hold only if $\sigma_0^1=\sigma_0^2$.)
  If the strict transforms intersect different parts of the graph, one has
  $(C_1\circ C_2)= \ell_1\cdot\ell_2\cdot m_{\sigma_0^1\sigma_0^2}$. Moreover,
  from the matrices $(m_{\sigma\delta})$ above
  one can see that, for fixed up to symmetry
  $\sigma_0^1$ and $\sigma_0^2$, the intersection
  number $m_{\sigma_0^1\sigma_0^2}$ for the
  components
  $\sigma_0^1$ and $\sigma_0^2$ from the same part is strictly larger than the one for the
  components from different parts.
 \end{proof}

\section{The case of several valuations}
The idea of the proofs of Theorems~\ref{theo:Main_curves} and~\ref{theo:Main_divisorial} is the same as
of Theorems~3 and~4 in \cite{MZ}.
Moreover, the proof of Theorem~\ref{theo:Main_divisorial} is literally the same modulo one remark related with the symmetry
of the minimal resolution graph of the surface singularity: see at the end of the section.
This is explained by the fact that the ``projection
formula'' (the equation connecting the Weil--Poincar\'e series of a collection of valuations with
the one for the collection with one valuation excluded) is much simpler for a collection of divisorial valuations. In this case the Weil--Poincar\'e series of the smaller collection is obtained from the other one simply by putting the value 1
for the corresponding variable $t_i$. Thus the Weil--Poincar\'e series of the smaller collection is
determined by the same series for the larger one.
This is not the case, in general, for a collection
of curve valuations. In this case the projection formula includes a certain factor for the excluded
valuation which, in general, cannot be directly obtained from the initial Weil--Poincar\'e series.
(This also explains the fact that, in~\cite{MZ}, the proof for divisorial valuations is much shorter than
for curve ones.) For the curve case (Theorem~\ref{theo:Main_curves}) the proof contains
some differences. Therefore we include some parts of it.

Let $C$ be an irreducible curve
germ in the surface $(S,0)$ and let
$\pi:(\XX,\DD)\to(S,0)$ be an embedded resolution of the curve
$(C,0)\subset(S,0)$.
Let $\Gamma$ be the minimal resolution graph of $C$, i.~e.,  the dual graph of
$\pi$.
Let $\b{C}$ be the total transform of the curve $C$ in $\XX$.
One has
$$
\b{C} = \w{C} + \sum_{\sigma\in \Gamma} m_{\sigma} E_\sigma\;,
$$
where $\w{C}$ is the strict transform of the curve $C$. The rational numbers
$\{m_{\sigma} : \sigma\in \Gamma\}$ are uniquely determined by the system of linear
equations
\begin{equation}\label{eq:M}
\{ \w{C} \cdot E_{\alpha} + \sum_{\sigma\in \Gamma} m_{\sigma} \; E_{\sigma}\cdot
E_{\alpha} = 0 \ : \ E_{\alpha}\in \cal{D} \}\; .
\end{equation}
(Each equation is a consequence of the fact
$\b{C} \cdot E_{\alpha}=0$,
see~\cite[Equation
(1)]{MUM}).

Let $\rho\in \Gamma$ be an end and let $\Delta = \{\rho = \alpha_0, \alpha_1,\ldots.
\alpha_s = \sigma\}$
be the corresponding dead arc in $\Gamma$ (i.e. the minimal connected subgraph of
$\Gamma$ such that $\sigma$ is a star vertex). Then one has:

\begin{lemma}\label{lemma:0}
If $-E_{\alpha_i}^2 \neq 1$ for $i=0,\ldots, s-1$, then there exists an integer
$N>1$, independent of the branch $C$, such that
$m_{\sigma} = N m_{\rho}$.
\end{lemma}

\begin{proof}
Equation~(\ref{eq:M}) for $E_{\rho}$ gives
$m_{\alpha_1} = (-E_{\rho}^2)m_{\rho} = N_1 m_{\rho}$ with an integer $N_1>1$.
Again, the same equation for $E_{\alpha_i}$ gives
\begin{align*}
m_{\alpha_{i+1}} & = (-E_{\alpha_i}^2) m_{\alpha_i}- m_{\alpha_{i-1}}
 = -E_{\alpha_{i}}^2 N_i m_{\rho} -  N_{i-1}m_{\rho}
=\\
& = (-E_{\alpha_{i}}^2 N_i -  N_{i-1}) m_{\rho} =
N_{i+1} m_{\rho}\; .
\end{align*}
Since  $E_{\alpha_{i}}^2\neq -1$, one has
$N_{i+1}\ge 2 N_{i}-N_{i-1} > N_{i}$ provided we assume
$N_i> N_{i-1}$ by induction.
\end{proof}

\medskip

Let $C=\bigcup_{i=1}^r C_i$ be a reducible (that is, $r>1$) curve germ in the
surface $(S,0)$ and let
$\pi:(\XX,\DD)\to(S,0)$ be the minimal embedded resolution of the curve
$(C,0)\subset(S,0)$.
Let $\Gamma$ be the minimal resolution graph of $C$, i.~e.,  the dual graph of
$\pi$. Let $\tau(i)$ be
the vertex of $\Gamma$ such that the component $E_{\tau(i)}$ of the exceptional
divisor $\DD$ intersects
the strict transform $\widetilde{C}_i$ of the curve $C_i$ and let $m_{\sigma}^i:=m_{\sigma\tau(i)}$.
One has
$\mm_{\sigma}=(m_{\sigma}^1, \ldots, m_{\sigma}^r)$. The reason (somewhat
psychological) for that is the fact that,
for a multi-exponent of a term of the Poincar\'e series $P_C(\seq t1r)$ or of a
factor of its decomposition, one
knows its components $m_{\sigma}^1$, \dots, $m_{\sigma}^r$, but does not know the
vertex $\sigma$.
One can say that our aim is to find vertices $\tau(i)$ corresponding to the curve.

\medskip

Let $\b{C}_k$ ($k=1,\ldots,r$) be the total transform of the curve $C_k$ in $\XX$.
One has
$$
\b{C}_k = \w{C}_k + \sum_{\sigma\in \Gamma} m_{\sigma}^k E_\sigma\;,
$$
where $\w{C}_k$ is the strict transform of the curve $C_k$.

Let us fix a pair of branches $C_i$ and $C_j$ and let $q : \Gamma\to \Q$ be the
function defined by
$q(\alpha) = m_{\alpha}^j/m_{\alpha}^i$ for $\alpha\in \Gamma$.
One has the following statement.

\begin{lemma}\label{lemma:4}
Let $E_{\alpha}$ be a component of the exceptional divisor  $\DD$ such that
$\w{C}_i\cdot E_{\alpha}=0$ and let $\{\seq {\rho}1s\}\subset \Gamma$
be the set of all vertices connected by an edge with $\alpha$.
Let us assume that either $\w{C}_j$ intersects $E_{\alpha}$ or
there exists $\rho_{i_0}$ such that $q(\rho_{i_0}) > q(\alpha)$. Then
there exists $\rho_k$ such that $q(\alpha)> q(\rho_k)$.
\end{lemma}

\begin{proof}
Assume that $q(\rho_k)\ge q(\alpha)$ for any $k=1,\ldots, s$. Applying (\ref{eq:M}) to
$C_j$ and $C_i$ one gets:
\begin{align*}
0 &= \w{C}_j \cdot E_{\alpha} + m_{\alpha}^j E_{\alpha}^2 + \sum_{k=1}^s m_{\rho_k}^j
\ge
\\
&\ge  \w{C}_j \cdot E_{\alpha} + m_{\alpha}^j E_{\alpha}^2 + \sum_{k=1}^s
q(\alpha) m_{\rho_k}^i = \\
& =  \w{C}_j \cdot E_{\alpha} + q(\alpha) ( m_{\alpha}^i E_{\alpha}^2 + \sum_{k=1}^s
m_{\rho_k}^i) = {\widetilde{C}}_j \cdot E_{\alpha} \ge 0
\end{align*}
The inequality is strict if $\w{C}_j \cdot E_{\alpha}>0$ or if
there exists $i_0$ such that $q(\rho_{i_0})>q(\alpha)$. This implies the statement.
\end{proof}

Let $[\tau(j), \tau(i)]\subset \Gamma$ be the (oriented) geodesic from
$\tau(j)$ to $\tau(i)$ and let
$\{\Delta_p\}$, $p\in\Pi$, be the connected components of
$\Gamma\setminus [\tau(j),\tau(i)]$. For each $p\in\Pi$ there exists a
unique $\rho_p\in [\tau(j),\tau(i)]$ connecting $\Delta_p$ with
$[\tau(j),\tau(i)]$, i.~e., such that
$\Delta^*_{p} = \Delta_p\cup \{\rho_p\}$ is connected.

\begin{proposition}\label{prop:P1}
With the previous notations,
one has:
\begin{enumerate}
\item The function $q$ is strictly decreasing along the geodesic
$[\tau(j), \tau(i)]$.
\item For each $p\in\Pi$, the function $q$ is constant on $\Delta^*_p$.
\end{enumerate}
\end{proposition}

\begin{proof}
Let $\alpha$ and $\beta$ be two vertices of $\Gamma$ connected by an edge and let
$q(\alpha)>q(\beta)$. Lemma~\ref{lemma:4} permits to construct a maximal
sequence $\alpha_0, \alpha_1, \ldots, \alpha_k$ of consecutive vertices starting with $\alpha$ and
$\beta$ (i.~e., $\alpha_0 = \alpha$, $\alpha_1=\beta$) such that $q(\alpha_i)>q(\alpha_{i+1})$.
(We will call a sequence of this sort {\em a decreasing path}. If the inequality is in the other direction,
the path will be called {\em increasing}.)
The maximality means that either $\alpha_k$ is a deadend of $\Gamma$ or $\w{C}_i\cdot E_{\alpha_k}\neq 0$.
If $\alpha_k$ is a deadend, $\alpha_{k-1}$ is the only vertex connected with $\alpha_k$ and
Lemma~\ref{lemma:4} implies that $q(\alpha_k)=q(\alpha_{k-1})$. Therefore the
constructed path finishes by the vertex $\alpha_k=\tau(i)$.
Note that, if $\alpha\in [\tau(j), \tau(i)]$ and $\beta\notin [\tau(j),\tau(i)]$, the end of
a maximal decreasing (or increasing) path has to finish at a deadend and therefore $q(\alpha)=q(\beta)$.
In particular, this implies that
the function $q$ is constant on each connected set $\Delta^*_p$.

Assume that $\tau(i)\neq \tau(j)$. Lemma~\ref{lemma:4} implies that there exists
a vertex $\alpha_1$ connected with $\tau(j)$ such that
$q(\tau(j)) > q(\alpha_1)$. Therefore the maximal decreasing path starting with
$\tau(j)$ and $\alpha_1$ coincides with the geodesic $[\tau(j), \tau(i)]$.
\end{proof}

\begin{remark}
Let $\rho\in \Gamma$ be an end and let $\Delta = \{\rho = \alpha_0, \alpha_1,\ldots.
\alpha_s = \sigma\}$  be the corresponding dead arc in $\Gamma$. In this case
Proposition~\ref{prop:P1} implies that the ratio $m^{i}_{\alpha}/m^j_{\alpha}$ is
constant for $\alpha\in \Delta$ and for any pair $i,j\in  \{1,\ldots,r\}$.
In fact, from Lemma~\ref{lemma:0}, one can easily deduce that
$\mm_{\sigma} = N \mm_{\rho}$ for an integer $N>1$, in particular,
$\mm_{\sigma} > \mm_{\rho}$.
\end{remark}


\begin{Proof} of Theorem~\ref{theo:Main_curves}.
We have to show that the Weil-Poincar\'e series $P^{W}_{\{C_i\}}(\tt)$ determines the
minimal
resolution graph $\Gamma$ of $C$. In the case under consideration one has a
projection formula different of the one for divisorial valuations.

Let $i_0\in \{1,\ldots,r\}$. The A'Campo type formula~(\ref{eq:def_W_Poincare}) for
$P^{W}_{\{C_i\}}(\tt)$ implies that
\begin{equation}
\label{eq:E1}
P^W_{\{C_i\}}(\tt)_{|_{t_{i_0}=1}} = P^W_{C\setminus \{C_{i_0}\}}(t_1, \ldots,
t_{i_0-1},
t_{i_0 +1}, \ldots, t_r) \cdot (1-\tt^{\mm_{\tau(i_0)}})_{|_{t_{i_0}=1}} \; .
\end{equation}
Applying~(\ref{eq:E1}) several times one gets
\begin{equation}
\label{eq:E2}
P^W_C(\tt)_{\vert_{t_j=1\text{ for }j\ne i_0}} = P^W_{C_{i_0}}(t_{i_0}) \cdot
\prod_{i\neq i_0}
(1-t_{i_0}^{m_{\tau(i)}^{i_0}}) \; .
\end{equation}
Pay attention to the fact that
$m_{\tau(i)}^{i_0} = m_{\tau(i_0)}^{i}$ and therefore the series
$P^W_{C_{i_0}}(t_{i_0})$ can be determined from the Weil-Poincar\'e series $P_C(\tt)$
if one knows the multiplicity
$\mm_{\tau(i_0)}$. The strategy of the proof follows the steps from~\cite{FAOM}
(see also~\cite{MNach}):
\begin{enumerate}
\item[1)] To detect an index $i_0$ for which one can find the corresponding
multiplicity $\mm_{\tau(i_0)}$
from the A'Campo type formula for $P^W_C(\tt)$. Then Theorem 2 and
equation~(\ref{eq:E2}) permit to recover the
minimal resolution graph $\Gamma_{i_0}$ of the curve
$C_{i_0}$. Equation~(\ref{eq:E1}) gives the possibility to compute the Poincar\'e series
$P^W_{C\setminus \{C_{i_0}\}}(t_1, \ldots, t_{i_0-1}, t_{i_0 +1}, \ldots, t_r)$ of
the curve $C\setminus \{C_{i_0}\}$.
By induction one can assume that the resolution graph $\Gamma^{i_0}$ of the curve
$C\setminus \{C_{i_0}\}$ is known.
Moreover, Lemma~\ref{lemma:two} implies that,
for each $j$, the
multiplicity $\mm_{\tau(i_0)}$ determines
whether the vertices of the minimal resolution
graph coresponding to the curves $C_{i_0}$ and $C_j$ are on the same part from those exchanged by symmetries or on different ones.

\item[2)] To determine the separation vertex of the curves $C_{i_0}$ and $C_j$ for $j\neq i_0$ in order
to join the graphs $\Gamma_{i_0}$ and $\Gamma^{i_0}$ to obtain the resolution graph $\Gamma$.
\end{enumerate}

The second step literally repeats the same one in the proof of Theorem~3 in~\cite{MZ}. Therefore
we omit an analysis of it here.

\medskip

Proposition~\ref{prop:P1} implies that, for any fixed $i_0$ and
for any $j\neq i_0$ and $\sigma\in \Gamma$, one has
$m_{\sigma}^j/m_{\sigma}^{i_0} \ge m_{\tau(i_0)}^j/m_{\tau(i_0)}^{i_0}$. Therefore one has
$$
\frac {1}{m_{\sigma}^{i_0}} \mm_{\sigma} \ge
\frac {1}{m_{\tau(i_0)}^{i_0}} \mm_{\tau(i_0)} \; .
$$

Let $P^W_C(\tt) = \prod_{k=1}^{p} (1-\tt^{\nn_k})^{s_k}$ be the Weil-Poincar\'e
series of the curve
$C$, where $s_k\neq 0$ for all $k$.
Note that the only case in which $P^W_C(\tt)=1$ corresponds to the singularity
$\A_k$ and two branches $C_1$, $C_2$ in such a way that they are curvettes at the end
points named $1$ and $k$
of the dual graph of $\A_k$. In the sequel we omit this trivial situation.
For $i\in I=\{1,\ldots,r\}$ let
$\kappa: I \to \{1,\ldots, p\}$ be the map defined by
$k=\kappa(i)$ be such that $s_k>0$ and
\begin{equation}\label{eq*}
\frac{1}{n_{j}^{i}} \nn_{j} \ge
\frac{1}{n_{k}^{i}} \nn_{k}
\end{equation}
for all $j$.
Note that if the inequalities~(\ref{eq*}) hold for
$\nn_k= \mm_{\rho}$, $\rho$ a deadend of $\Gamma$, then (see
Proposition~\ref{prop:P1})
one has the same condition for $\sigma$, the star vertex of $\Gamma$ more
close to $\rho$.
Let $E = \kappa (I) \subset \{1,\ldots, p\}$ be the set of indices $k$ such that
$k=\kappa(i)$ for some $i\in \{1,\ldots, r\}$ and for $k\in E$ let $A(k)\subset
\{1,\ldots, r\}$ denote the set of indices $i$ such that $k=\kappa(i)$.
Note that
$A(k)$ contains all the indices $i\in \{1,\ldots,r\}$ such that
$\nn_k = \mm_{\tau(i)}$. Let $B(k)$ be the subset of such indices.
Our aim is to show that one can find $k\in E$ such that $B(k)\neq \emptyset$.

\medskip

\begin{proposition}
Let us assume that $\# E\ge 2$. Then there exists $k\in E$ and $i\in A(k)$ such that
\begin{enumerate}
\item
$n^i_{k} \ge n^j_{k}$ for any $j\in A(k)$
\item
$n^j_k \le n^i_{k'}$ for any  $k'\in E$, $k'\neq k$, and $j\in A(k')$.
\end{enumerate}
Moreover, for any pair $(k,i)$ satistying conditions 1) and 2) above, one has that
$i\in B(k)$ and 
therefore $B(k)\neq \emptyset$.
\end{proposition}

\begin{proof}
Let $j\in A(k)$, $j\notin B(k)$. One has
$$
\frac {1}{n_{k}^{j}}\nn_{k} > \frac{1}{m_{\tau(j)}^{j}} \mm_{\tau(j)}
$$
and therefore $\chi(\oE_{\tau(j)})=0$. This
implies that $\tau(j)$ is connected with only one vertex in $\Gamma$ (plus
the  arrow corresponding to $\w{C}_j$), i.~e., $\tau(j)$ is a deadend of the
resolution graph of the curve $C\setminus\{C_j\}$.
Let
$\sigma\in \Gamma$ be such that $\nn_k= \mm_\sigma$ and assume that
$i\in B(k)\neq \emptyset$. In this case
$\nn_k = \mm_{\tau(i)}$ and one has 
$n^i_{k}=m^i_{\tau(i)} = N m^i_{\tau(j)}$,
with an integer $N>1$,
being $\tau(j)$ a deadend of the dual graph of $C\setminus \C_j\}$
(see~\ref{lemma:0}). In particular
$n^i_k > m^i_{\tau(j)} = m^j_{\tau(i)} = n^j_{k}$.
Notice that if $i,s\in B(k)$ one has also that $\nn_k = \mm_{\tau(s)}$ and 
therefore $n^i_{k} = n^s_k$.
As a consequence, if $B(k)\neq \emptyset$ one can fix an index $i(k)\in B(k)$ taking a
maximal one of the set $\{n^i_k : i\in A(k)\}$ as in the condition 1) of the statement.

Let $k'\in E$, $k'\neq k$, and $j\in A(k')$. If $j\in B(k')$, then
$\nn_{k'} = \mm_{\tau(j)}$ and therefore
$n^{i(k)}_{k'} = m^{i(k)}_{\tau(j)} = m^j_{\tau(i(k))} = n^j_{k}$.
Otherwise $j\in A(k')\setminus B(k')$ and one has that
$n^j_k = m^{j}_{\tau(i(k))} = m^{i(k)}_{\tau(j)}$ and also
$n^{i(k)}_{k'} = N m^{i(k)}_{\tau(j)}$ for 
a positive integer $N$. As a
consequence
$n^{i(k)}_{k'} > n^j_k$ and 
therefore $i(k)$ satisfies the second requirement of the
proposition.

In order to finish the proof one has to prove that $B(k)\neq \emptyset$ for some
$k\in E$.
Let us assume that $B(k)=\emptyset $ for some $k\in E$ and
$\nn_k= \mm_\sigma$, for some $\sigma\in \Gamma$ with $v(\sigma)\ge 3$ (here
$v(\sigma)$ is the valency of the vertex $\sigma$).
Let $\pi': (\XX',\DD' )\to (S,0)$ be the pre-resolution of $C$ and let
$\Gamma'$ be the resolution graph of $\pi'$. For any $j\in A(k)$ one has that
$\tau(j)$ is an end of $\Gamma\setminus \{\w{C}_j\}$. We will distinguish two cases:

Let us assume that $\sigma\notin \Gamma'$. In this case also $\tau(j)\notin \Gamma'$
is a deadend of the dual resolution graph of the curve $C\setminus \{C_j\}$ and
$A(k)=\{j\}$.
In particular,
the vertex $\sigma$ appears after $\tau(j)$ in the resolution process of a certain
branch $C_i$,
$i\neq j$, which is not a curvette at $E_{\sigma}$. It is clear that in this case
$(1/n^i_{k})\nn_{k} > (1/m^i_{\tau(i)})\mm_{\tau(i)}$ and
$k'=\kappa(i)$ provides a new element $k'\in E$. Using this new element $k'$ we can
repeat the argument up to the moment when we reach $e\in E$ such that $B(e)\neq \emptyset$ (note that
$\nn_k < \nn_{k'}$).

Assume that  $\sigma\in \Gamma'$. If there exists an irreducible component $C_i$ such
that its strict transform by $\pi'$ intersects $\DD'$ at $E_\sigma$ one can repeat
the same argument of the above case for $k'=\kappa(i)$ and so there exists
an element $e\in E$ such that $B(e)\neq \emptyset$.
Otherwise $\sigma\in \Gamma'$ must be a star vertex of $\Gamma'$ and all the
elements of $A(k)$ correspond to curvettes at some of the end points of $\Gamma'$.
However, there is (at most) only one vertex in $\Gamma'$ with such conditions:
the vertex named $k-2$ if $S=\D_k$, the vertex named $3$ for $\E_6, \E_7$ and $\E_8$
and no one for the case $\A_k$. Being $\# E\ge 2$ in order to finish it suffices to
take a new $k'\in E$, $k'\neq k$.
\end{proof}

In order to finish the proof of the Theorem it remains only to treat
the case when $\# E =1$, i.e. $k\in E$ is such that $A(k)=\{1,\ldots,r\}$. First of all we
will identify the cases when $B(k)=\emptyset$.
Taking into account the discussion in the proof of the Proposition above about the
indices $j\in A(k)\setminus B(k)$, the only possibility to have this situation is
when  $\nn_k = \mm_\sigma$ for $\sigma\in \Gamma'$ being a star vertex and 
all the branches are curvettes at some end points of $\Gamma'$ and no two of them
are in the same vertex. This situations can be described (and so detected) one by
one for each of the singularities $\D_k$, $\E_6, \E_7, \E_8$ (note that $\A_k$
does not appear in this situation). The table below describe all the possible choices
of the ends and the corresponding Weil-Poincar\'e series:

$$
\begin{matrix}
\D_k & \{1,k\}  & \mapsto & (1- t_1 t_2^{(k-2)/2}) (1-t_1^{1/2} t_2^{(k-2)/4})^{-1}
\\
& \{k-1,k\} & \mapsto & (1- t_1^{(k-2)/2} t_2^{(k-2)/2}) (1-t_1^{1/2} t_2^{1/2})^{-1}
\\
& \{1,k-1,k\} & \mapsto & (1- t_1 t_2^{(k-2)/2} t_3^{(k-2)/2}) \\
\E_6 & \{1,4\}  & \mapsto & (1- t_1^2 t_2^{3}) (1-t_1^{2/3} t_2^{1})^{-1} \\
 & \{1,6\}  & \mapsto & (1- t_1^2 t_2^{2}) (1-t_1 t_2)^{-1} \\
 & \{1,4,6\}  & \mapsto & (1- t_1^2 t_2^{3}t_3^2) \\
\E_7 & \{1,4\}  & \mapsto & (1- t_1^4 t_2^{6}) (1-t_1^{} t_2^{3/2})^{-1} \\
 & \{1,7\}  & \mapsto & (1- t_1^4 t_2^{3}) (1-t_1^2 t_2{3/2})^{-1} \\
 & \{4,7\}  & \mapsto & (1- t_1^6 t_2^{3}) (1-t_1^2 t_2^{})^{-1} \\
 & \{1,4,7\}  & \mapsto & (1- t_1^4 t_2^{6}t_3^3) \\
\E_8 & \{1,4\}  & \mapsto & (1- t_1^{10} t_2^{15}) (1-t_1^{2} t_2^{3})^{-1} \\
 & \{1,8\}  & \mapsto & (1- t_1^{10} t_2^{6}) (1-t_1^{5} t_2{^3})^{-1} \\
 & \{4,8\}  & \mapsto & (1- t_1^{15} t_2^{6}) (1-t_1^{5} t_2^2)^{-1} \\
 & \{1,4,8\}  & \mapsto & (1- t_1^{10} t_2^{15}t_3^6) \\
\end{matrix}
$$

Now, if the Weil-Poincar\'e series is not one of the above, then $B(k)\neq\emptyset$
and the we can choose an index $i\in B(k)$ as in the Proposition above.
This finishes the proof of Theorem~\ref{theo:Main_curves}.
\end{Proof}

For the proof of Theorem~\ref{theo:Main_divisorial} (an analogue of Theorem~\ref{theo:Main_curves} for divisorial valuations) the projection formula permits to reduce (to split) the case of $r$ valuations to the cases of one valuation $v_1$ and $(r-1)$ remaining valuations. For the $\E_8$-singularity in~\cite{MZ} this finishes the proof (due to the absence of symmetries of the $\E_8$ graph). In the case under consideration the
multiplicity $\mm_{1i}$ determines
whether the vertices of the minimal resolution
graph coresponding to the divisorial valuations $v_1$ and $v_i$ are on the same part from those exchanged by symmetries or on different ones.

Addresses:

A. Campillo and F. Delgado:
IMUVA (Instituto de Investigaci\'on en
Matem\'aticas), Universidad de Valladolid,
Paseo de Bel\'en, 7, 47011 Valladolid, Spain.
\newline E-mail: campillo\symbol{'100}agt.uva.es, fdelgado\symbol{'100}uva.es

S.M. Gusein-Zade:
Moscow State University, Faculty of Mathematics and Mechanics, Moscow, GSP-1, 119991, Russia.
\newline E-mail: sabir\symbol{'100}mccme.ru


\begin{thebibliography}{12}



\bibitem{PSMI} Gusein-Zade S.M., Delgado F., Campillo A.
Integrals with respect to the Euler characteristic over the space of functions and the Alexander polynomial.
Proceedings of the Steklov Institute of Mathematics, 2002, v.238, 134--147.

\bibitem{Duke} Campillo A., Delgado F., Gusein-Zade S.M.
The Alexander polynomial of a plane curve singularity via the ring of functions on it.
Duke Math. J., v.117 (2003), no.1, 125--156.

\bibitem{IJM} Campillo A., Delgado F., Gusein-Zade S.M.
The Alexander polynomial of a plane curve singularity and integrals with respect to the Euler characteristic.
International Journal of Mathematics, v.14 (2003), no.1, 47--54.

\bibitem{Invent} Campillo A., Delgado F., Gusein-Zade S.M.
Poincar\'e series of a rational surface singularity.
Invent. Math., v.155 (2004), no.1, 41--53.

\bibitem{CMH} Campillo A., Delgado F., Gusein-Zade S.M.
Poincar\'e series of curves on rational surface singularities.
Comment. Math. Helv., v.80 (2005), no.1, 95--102.

\bibitem{FAA} Gusein-Zade S.M., Delgado F., Campillo A.
Universal abelian covers of rational surface singularities, and multi-index filtrations.
Funktsional. Anal. i Prilozhen. v.42 (2008), no.2, 3--10 (in Russian); translated in
Funct. Anal. Appl. v.42 (2008), no.2, 83--88.

\bibitem{FAOM} Campillo A., Delgado F., Gusein-Zade S.M.
The Poincar\'e series of divisorial valuations in the plane defines the topology of the set of divisors.
Funct. Anal. Other Math., v.3 (2010), no.1, 39--46.

\bibitem{MNach} Campillo A., Delgado F., Gusein-Zade S.M.
On the topological type of a set of plane valuations with symmetries.
Math. Nachr., v.290 (2017), no.13, 1925--1938.

\bibitem{MZ} Campillo A., Delgado F., Gusein-Zade S.M.
Math. Z., v.294 (2020), no.1-2, 593--613.

\bibitem{IJM-2010} Campillo A., Delgado F., Gusein-Zade S.M., Hernando F.
Poincar\'e series of collections  of plane valuations.
International Journal of Mathematics, v.21 (2010), no.11, 1461--1473.

\bibitem{CDK} Campillo A., Delgado F., Kiyek, K.
Gorenstein property and symmetry for one-dimensional local Cohen-Macaulay rings.
Manuscripta Math., v.83 (1994), no.3-4, 405--423.

\bibitem{DG} Delgado F., Gusein-Zade S.M.
Poincar\'e series for several plane divisorial valuations.
Proc. Edinb. Math. Soc. (2), v.46 (2003), no.2, 501--509.



\bibitem{MUM} Mumford D.
The topology of normal singularities of an algebraic surface and a criterion for simplicity.
Publ. Math. de l'IH\'ES, v.9 (1961), 5--22.

\bibitem{Nemethi} Nemethi A.
Poincar\'e series associated with surface singularities. Singularities I, 271--297, Contemp. Math., 474, Amer. Math. Soc., Providence, RI, 2008.


\bibitem{Wall} Wall C.T.C. Singular points of plane curves. London Math. Soc. Student
Texts no.63. Cambridge University Press, 2004.

\bibitem{Yamamoto}
Yamamoto M. Classification of isolated algebraic singularities by their Alexander polynomials.
Topology, v.23, no.3 (1984), 277--287.

\end{thebibliography}
\end{document}